\documentclass[10pt, reqno]{amsart}
\usepackage{graphicx, amssymb, amsmath, amsthm, color, slashed, cite}
\numberwithin{equation}{section}

\usepackage{hyperref}

\let\Im=\undefined\DeclareMathOperator*{\Im}{Im}

\newcommand{\R}{\mathbb{R}}

\newcommand{\eps}{\varepsilon}

\newcommand{\F}{\mathcal{F}}

\newtheorem{theorem}{Theorem}[section]
\newtheorem*{prop}{Proposition}
\newtheorem{lemma}[theorem]{Lemma}

\newtheorem{proposition}[theorem]{Proposition}

\theoremstyle{definition}

\newtheorem{remark}[theorem]{Remark}

\theoremstyle{remark}

\newcommand{\qtq}[1]{\quad\text{#1}\quad}

\begin{document}

\title[Inhomogeneous NLS]{Scattering for the non-radial inhomogeneous NLS}

\author[C. Miao]{Changxing Miao}
\address{Institute for Applied Physics and Computational Mathematics, Beijing, China}
\email{miao\_changxing@iapcm.ac.cn}

\author[J. Murphy]{Jason Murphy}
\address{Missouri University of Science and Technology, Rolla, MO, USA}
\email{jason.murphy@mst.edu}

\author[J. Zheng]{Jiqiang Zheng}
\address{Institute for Applied Physics and Computational Mathematics, Beijing, China}
\email{zhengjiqiang@gmail.com}

\begin{abstract} We extend the result of Farah and Guzm\'an  \cite{FG} on scattering for the $3d$ cubic inhomogeneous NLS to the non-radial setting.  The key new ingredient is a construction of scattering solutions corresponding to initial data living far from the origin. 
\end{abstract}

\maketitle

\section{Introduction}  

We consider the $3d$ focusing cubic inhomogeneous nonlinear Schr\"odinger equation:
\begin{equation}\label{inls}
\begin{cases}
(i\partial_t + \Delta) u + |x|^{-b} |u|^2 u=0 , \\
u|_{t=0}=u_0\in H^1(\R^3),
\end{cases}
\end{equation}
where $0<b<\tfrac12$.  This is an $\dot H^{s_c}$-critical problem, where $s_c=\tfrac{1+b}{2}\in(\tfrac12,\tfrac34)$. In \cite{FG}, the authors established a scattering result for initial data below the ground state threshold.  In particular, denoting by $Q$ the ground state solution to the equation
\[
\Delta Q - Q + |x|^{-b} Q^3 = 0
\]
and the conserved mass and energy of solutions by
\[
M[u]=\int |u|^2\,dx,\quad E[u] = \int \{\tfrac12 |\nabla u|^2 -\tfrac14 |x|^{-b} {|u|^4}\}\,dx,
\]
the authors of \cite{FG} proved the following:

\begin{theorem}\label{T:FG} Suppose $u_0\in H^1$ is radial and satisfies 
\begin{equation}\label{below1}
E[u_0]^{s_c} M[u_0]^{1-s_c} < E[Q]^{s_c} M[Q]^{1-s_c} 
\end{equation}
and
\begin{equation}\label{below2}
 \| \nabla u_0\|_{L^2}^{s_c}\|u_0\|_{L^2}^{1-s_c} < \| \nabla Q\|_{L^2}^{s_c}\|Q\|_{L^2}^{1-s_c}. 
\end{equation}
Then the corresponding solution $u(t)$ to \eqref{inls} is global in time and scatters, that is, there exist $u_\pm\in H^1$ so that
\[
\lim_{t\to\pm\infty}\|u(t) - e^{it\Delta}u_\pm\|_{H^1} = 0.
\]
\end{theorem}

The purpose of this note is to extend Theorem~\ref{T:FG} to the non-radial setting:

\begin{theorem}\label{T} Theorem~\ref{T:FG} holds without the radial restriction on $u_0$.
\end{theorem} 

In fact, as we will describe below, with the addition of one key ingredient the arguments of \cite{FG} are already sufficient to obtain the non-radial result.  The missing ingredient in \cite{FG} is a method for producing scattering solutions to \eqref{inls} corresponding to initial conditions living far from the origin.  We will prove such a result in Proposition~\ref{P}.  In order to prove Theorem~\ref{T}, we will then simply walk through the steps carried out in \cite{FG} and demonstrate how the addition of Proposition~\ref{P} allows for the inclusion of non-radial initial conditions.  

We would also like to point out that in \cite{FG2}, the authors extended the result of \cite{FG} to higher dimensions.  The methods we present should suffice to extend the results of \cite{FG2} to the non-radial case, as well.  We have opted to focus on the $3d$ cubic case to keep technical complications to a minimum.

Before proceeding to the proof, let us briefly discuss some background related to the inhomogeneous NLS, as well as some works that capitalize on results similar to Proposition~\ref{P}.

The model \eqref{inls}, along with some generalizations, has been the subject of recent mathematical interest; see, e.g.  \cite{GS1, GS2, GS3, FG, FG2, Dinh, Guzman}.  The specific result of \cite{FG} under discussion in this paper fits in the context of many recent results concerning sharp scattering thresholds (typically described in terms of a ground state solution) for focusing intercritical nonlinear Schr\"odinger equations.  Such results were first established for the standard power-type NLS (see \cite{DHR, HR, Guevara, CFX, AkahoriNawa, DM1, DM2, ADM}), although many extensions to related models are now available (see e.g. \cite{KMVZ, FG, Zheng, LMM, SWYZ, XZ, Arora, KVZ}).  Many of the works just cited, including the work of \cite{FG}, follow the `Kenig--Merle roadmap' of \cite{KM}, reducing the problem of scattering for arbitrary sub-threshold solutions to the preclusion of compact sub-threshold solutions. The reduction is carried out using concentration-compactness arguments, while the preclusion is typically achieved through virial arguments.  Beginning with the work of \cite{DM1}, there has also been a trend towards establishing sharp scattering results (for NLS and related models) using technically simpler arguments that avoid concentration compactness entirely.  This typically requires a radial assumption on the initial data, which essentially provides enough compactness (via tools like radial Sobolev embedding) to run a virial argument for general sub-threshold solutions (see also \cite{DM2} for a non-radial result).  

Our interest in this paper is to employ some ideas coming from the study of dispersive equations with broken symmetries to obtain the non-radial analogue of the result of \cite{FG}.  Particularly relevant are the works \cite{LMM, KMVZ, KMVZZ2}, which consider the scattering problem for NLS with an inverse-square potential (i.e power-type NLS with $-\Delta$ replaced by $-\Delta+a|x|^{-2}$), and also proceed along the `Kenig--Merle roadmap'.  This model shares some similarities with the inhomogeneous NLS, in the sense that it retains a scaling symmetry but has a broken space translation symmetry. 

In \cite{LMM, KMVZ, KMVZZ2}, a key challenge arising from the broken translation symmetry appears in the construction of compact blowup solutions.  As we will discuss in Section~\ref{S:Proof} below, this construction relies first on a linear profile decomposition for a sequence of initial data, and then subsequently on a `nonlinear profile decomposition' obtained by constructing (scattering) nonlinear solutions associated to each profile.  The difficulty arises from the fact that each profile comes with some translation parameters $x_n$, which will either vanish identically or satisfy $|x_n|\to\infty$.  In particular, since the translation symmetry is broken, one cannot construct solutions for profiles with $|x_n|\to\infty$ by simply solving the equation with data given by the profile and then incorporating the translation. The resolution in \cite{LMM, KMVZ, KMVZZ2} comes from the observation that in the regime $|x|\to \infty$, the effect of the potential $a|x|^{-2}$ becomes weak.  Therefore, one can construct an approximate solution to the full problem by using a solution to the standard NLS (i.e. with no potential), and then appealing to a stability result to produce the true desired solution.  In this sense, one finds the standard NLS `embedded' inside the model of the NLS with inverse-square potential in the regime $|x|\to\infty$. 

In \cite{FG}, the authors' restriction to radial initial data for \eqref{inls} means that the translation parameters vanish from the problem entirely, i.e. one can always take $x_n\equiv 0$.  In particular, this allows them to avoid the issue described above entirely.  In fact, a careful study of \cite{FG} reveals that this is the \emph{only} job of the radial assumption in that work (i.e. there is no use of radial Sobolev embedding, radial Strichartz estimates, or any other radial tools).  The key observation in the present paper is that one can remove the radial assumption provided one can exclude the possibility of $|x_n|\to\infty$ by some other means.  In particular, this can be achieved provided we can produce scattering solutions associated to any profile with diverging translation parameters. We achieve this in Proposition~\ref{P}.  To produce these scattering solutions, we use the same philosophy as described above.  This time, however, we observe that in the regime $|x|\to\infty$, the nonlinearity itself becomes weak, and hence solutions to \eqref{inls} should simply be approximated by solutions to the underlying linear Schr\"odinger equation.  Put differently, we find the underlying linear equation `embedded' inside \eqref{inls} in the regime $|x|\to\infty$.  For a more detailed explanation of the exact approximation we use, see Remark~\ref{Explain}.  

In Section~\ref{S:Proof} we will discuss how, once we have incorporated Proposition~\ref{P}, we can follow the rest of the arguments of \cite{FG} more or less verbatim to deduce the sub-threshold scattering theorem for arbitrary (i.e. non-radial) initial data.  We then carry out the proof of Proposition~\ref{P} in Section~\ref{S:P}. 

We would like to point out that the works \cite{LMM, KMVZ, KMVZZ2} are certainly not the first works to capitalize on the ideas just discussed.  In general, in the setting of dispersive equations with broken symmetries, one needs to understand the models that may be `embedded' in the full equation in various limiting scenarios.  We refer the reader to the following list of references, which is certainly not exhaustive, but hopefully serves to demonstrate the importance and flexibility of these ideas: \cite{KMVZ, LMM, KVZ, KVZ0, KKSV, KSV, IP1, IP2, IPS, Jao1, Jao2, KMVZZ2, KOPV, PTW, KMV}.

The rest of this paper is organized as follows:  In Section~\ref{S:notation}, we set up notation and collect a few preliminary results. In Section~\ref{S:Proof}, we present the proof of Theorem~\ref{T}, taking the main new ingredient Proposition~\ref{P} for granted.  Finally, in Section~\ref{S:P}, we prove Proposition~\ref{P}.

\subsection*{Acknowledgements} C.M. was supported by NFSC Grants 11771041 and 11831004.  J.M. was supported by a Simons Collaboration Grant.  J.Z. was supported by NSF Grant 11901041. 


\section{Notation and preliminaries}\label{S:notation}

We write $A\lesssim B$ to denote $A\leq CB$ for some $C>0$.  We also make use of the notation $a\pm$ to denote $a\pm\eps$ for some sufficiently small $\eps>0$.  We use the standard notation for Lebesgue space-time norms and Sobolev norms, e.g. $L_t^q L_x^r$ and $L_t^\infty H_x^1$. 

We employ the standard Littlewood--Paley projections $P_{\leq N}$.  These are defined as Fourier multipliers, with the multiplier corresponding to a smooth cutoff to the region $\{|\xi|\leq N\}$.  We need  only a few basic facts, e.g. the Bernstein estimate
\[
\| |\nabla|^s P_{\leq N} f\|_{L_x^2} \leq N^s \|f\|_{L_x^2} 
\]
and the fact that $P_{\leq N} f \to f$ strongly in $H^1$ as $N\to\infty$. 

To deal with the function $|x|^{-b}$ appearing in the nonlinearity, we have found it convenient to utilize Lorentz spaces, defined via the quasi-norms
\[
\|f\|_{L_x^{p,q}} = \bigl\| \lambda \bigl| \{x: |f(x)|>\lambda \} \bigr|^{\frac{1}{p}} \bigr\|_{L^q((0,\infty),\frac{d\lambda}{\lambda})} 
\]
for $1\leq p<\infty$ and $1\leq q\leq\infty$.  In particular $L^{p,p}=L^p$, while $L^{p,\infty}$ corresponds to the weak $L^p$ space.  In general we have the embedding $L^{p,q}\hookrightarrow L^{p,q'}$ for $q<q'$.  These spaces are natural in the context of \eqref{inls} since $|x|^{-b}\in L^{\frac{3}{b},\infty}(\R^3)$.

Many standard functional inequalities have analogues in Lorentz space (see e.g. \cite{Hunt, ONeil}).  For example, we have the H\"older inequality
\[
\|fg\|_{L^{p,q}} \lesssim \|f\|_{L^{p_1,q_1}}\|g\|_{L^{p_2,q_2}}
\]
for $1\leq p,p_1,p_2<\infty$ and $1\leq q,q_1,q_2\leq\infty$ satisfying $\tfrac{1}{p}=\tfrac{1}{p_1}+\tfrac{1}{p_2}$ and $\tfrac{1}{q}=\tfrac{1}{q_1}+\tfrac{1}{q_2}$.  We also have Young's convolution inequality
\[
\|f\ast g\|_{L^{p,q}} \lesssim \|f\|_{L^{p_1,q_1}}\|g\|_{L^{p_2,q_2}}
\] 
for the same range of exponents now satisfying $\tfrac{1}{p}+1=\tfrac{1}{p_1}+\tfrac{1}{p_2}$ and $\tfrac{1}{q}=\tfrac{1}{q_1}+\tfrac{1}{q_2}$. 

Using Young's inequality for Lorentz spaces, we can also establish a Lorentz-space version  of Sobolev embedding, which will be useful below. Indeed, writing $d\geq 1$ for the spatial dimension and recalling $\F[|x|^{s-d}]=c|\xi|^{-s}$ for $0<s<d$ (see e.g. \cite{Stein}), we have the estimate
\[
\| |\nabla|^{-s} f\|_{L^{r,q}} \sim \| |x|^{s-d}\ast f\|_{L^{r,q}} \lesssim \| |x|^{s-d}\|_{L^{\frac{d}{d-s},\infty}} \| f\|_{L^{p,q}}\lesssim \|f\|_{L^{p,q}}
\]
for $1<r,p<\infty$ satisfying $\tfrac{d}{r}=\tfrac{d}{p}-s$ and $1\leq q\leq\infty$. 

In \cite{FG}, Strichartz estimates for the linear Schr\"odinger equation are stated using the following notation.  One defines the region $\mathcal{A}_s$ (for $s\in\R$) to be the set of $(q,r)$ satisfying
\[
(\tfrac{6}{3-2s})^+ \leq r\leq 6^- \qtq{and}\tfrac{2}{q}+\tfrac{3}{r}=\tfrac{3}{2}-s. 
\]
When $s=0$, the endpoints are included.  The Strichartz norm is then defined by
\[
\|u\|_{S(\dot H^s)} = \sup_{(q,r)\in \mathcal{A}_s} \|u\|_{L_t^q L_x^r},
\]
with dual Strichartz norm given by
\[
\|F\|_{S'(\dot H^{-s})} = \inf_{(q,r)\in \mathcal{A}_{-s}} \|F\|_{L_t^{q'}L_x^{r'}}. 
\]
If no time interval is indicated, this refers to space-time norms over all of $\R\times\R^3$.  To denote the truncation to a finite time interval $I$, one writes $S(\dot{H}^s;I)$. 

One can characterize scattering versus `blowup' for \eqref{inls} according to the $S(\dot H^{s_c})$ norm (see e.g. \cite[Proposition~{3.1}]{FG}).  In particular, solutions may be extended as long as their $S(\dot H^{s_c})$ norm remains finite, and a global solution with finite $S(\dot H^{s_c})$-norm scatters to a free solution.  On the other hand, by `blowup' we typically refer to the blowup of the $S(\dot H^{s_c})$ norm, which may occur in finite or infinite time. 

We will utilize the following Strichartz estimates for the linear Schr\"odinger equation:
\begin{lemma}[Strichartz estimates] Let $e^{it\Delta}$ denote the free Schr\"odinger propagator.  Then
\[
\|e^{it\Delta} f\|_{S(\dot H^{s})}\lesssim \|f\|_{\dot H^s}. 
\] 
\end{lemma}
While inhomogeneous estimates hold as well (and are essential for the well-posedness and stability theory developed in \cite{FG}), in this note we will only need to make explicit use the homogeneous estimates stated above. 

We will rely fundamentally upon the following stability result, which apppears as Proposition~4.10 in \cite{FG} and was already essential in that work. 

\begin{lemma}[Stability]\label{L:Stab} Suppose $I$ is a time interval and $\tilde v$ is an approximate solution to \eqref{inls} on $I$, in the sense that
\[
i\partial_t\tilde v + \Delta \tilde v + |x|^{-b} |\tilde v|^2 \tilde v = e 
\]
for some function $e$ on $I$. Suppose that $\tilde u$ satisfies
\[
\|\tilde v\|_{L_t^\infty H_x^1(I\times\R^3)} + \|\tilde v\|_{S(\dot H^{s_c};I)} \leq C<\infty. 
\]
There exists $\eps_1=\eps_1(C)$ sufficiently small that if $u_0\in H^1$ satisfies
\[
\|u_0-\tilde v(0)\|_{H^1} < \eps
\]
and
\[
\|e\|_{S'(L^2;I)}+\|\nabla e\|_{S'(L^2;I)}+\|e\|_{S'(\dot H^{-s_c};I)} <\eps
\]
for some $0<\eps<\eps_1,$ then there exists a unique solution $u$ to \eqref{inls} on $I$ with $u(0)=u_0$, which satisfies
\[
\|u-\tilde v\|_{S(\dot H^{s_c};I)} \lesssim_C \eps
\]
and
\[
\|u\|_{S(\dot H^{s_c};I)} + \|u\|_{S(L^2;I)}+\|\nabla u\|_{S(L^2;I)} \lesssim_C 1. 
\]
\end{lemma}


\section{The proof of Theorem~\ref{T}}\label{S:Proof}

In this section, we review the  proof of Theorem~\ref{T:FG} from \cite{FG}.  As we proceed, we will introduce one new ingredient (Proposition~\ref{P}) into the argument and show how this ingredient allows for the treatment of non-radial initial conditions.  Thus we will be able to conclude that the extension to non-radial solutions (Theorem~\ref{T}) holds as well.  

The proof of Theorem~\ref{T:FG} proceeds by contradiction.  The first main step is to prove that if the theorem fails, one may construct a compact blowup solution living below the ground state threshold. The result may be stated as follows (cf. \cite[Propositions~6.4 and 6.5]{FG}): 

\begin{proposition}[Existence of a critical solution]\label{P:uc} Suppose Theorem~\ref{T} fails.  Then there exists a function $u_{c,0}\in H^1$ such that the corresponding solution $u_c$ to \eqref{inls} is global and uniformly bounded in $H^1$.  This solution is below the ground state threshold (that is, it satisfies \eqref{below1} and \eqref{below2}), blows up in both time directions (that is, $\|u_c\|_{S(\dot H^{s_c};\R_-)}=\|u_c\|_{S(\dot H^{s_c};\R_+)}=\infty$), and has a pre-compact orbit in $H^1$. 
\end{proposition}

With Proposition~\ref{P:uc} in hand, the authors conclude the proof of Theorem~\ref{T:FG} by carrying out a localized virial argument (see \cite[Theorem~7.3]{FG}).  In the context of \eqref{inls}, the virial identity is the following formula for the time derivative of the weighted momentum for solutions to \eqref{inls}:
\begin{equation}\label{virial}
\tfrac{d}{dt}\Im \int x\cdot \bar u\nabla u \,dx = c\int |\nabla u|^2 - |x|^{-b}|u|^4\,dx. 
\end{equation}
The variational characterization of the ground state $Q$ implies that for functions below the ground state threshold (i.e. obeying \eqref{below1} and \eqref{below2}), the right-hand side of the identity above is coercive, e.g. bounded below by a constant times the $\dot H^1$ norm. If the weighted momentum were uniformly bounded in time, then integrating the identity above over a sufficiently long time interval would lead to a contradiction (since the $\dot H^1$-norm is uniformly bounded below).  However, this quantity is not uniformly bounded due to the presence of the weight $x$.  The solution is to localize the argument above in space, say to $|x|\leq R$. Then the identity above no longer holds exactly, but instead contains error terms controlled by the following:
\[
\int_{|x|>R} |\nabla u(t,x)|^2 + R^{-2}|u(t,x)|^2 + R^{-b} |u(t,x)|^4\,dx. 
\]
As the orbit of $u$ is pre-compact in $H^1$, these error terms can be made small (say $\leq \eta$) uniformly in time provided $R=R(\eta)$ is chosen sufficiently large; furthermore, the localization of the quantity on the right-hand side of \eqref{virial} is still coercive (uniformly in time).  In particular, one can successfully carry out the scheme described above and derive a contradiction.  Indeed, one arrives at an inequality of the form 
\[
c(u)T\lesssim C(u)R(\eta) + \eta T\qtq{for any}T>0,
\] 
and the contradiction is obtained by choosing $\eta=\eta(u)$ sufficiently small and then $T$ sufficiently large. 

The discussion above shows that once Proposition~\ref{P:uc} is obtained, the proof can be completed.  Thus we turn our attention to the proof of Proposition~\ref{P:uc}.

By the well-posedness theory for \eqref{inls}, initial data obeying \eqref{below2} and with the quantity in \eqref{below1} small enough lead to global scattering solutions.  Thus, if Theorem~\ref{T:FG} (or Theorem~\ref{T}) fails, there is a critical value (denoted $\delta_c$ in \cite{FG}) for $E[u_0]^{s_c} M[u_0]^{1-s_c}$ that obeys $\delta_c<E[Q]^{s_c}M[Q]^{1-s_c}$ and separates the scattering and blowup regions for solutions obeying \eqref{below2}.  To prove Proposition~\ref{P:uc}, the scheme is then the following:
\begin{itemize}
\item[(i)] Construct a sequence of initial conditions $u_{n,0}$ obeying \eqref{below2} and satisfying $M[u_{n,0}]^{1-s_c}E[u_{n,0}]^{s_c}\to \delta_c,$ with corresponding solutions $u_n$ blowing up their space-time norms as $n\to\infty$.
\item[(ii)] Prove that $u_{n,0}$ converges along a subsequence in $H^1$ to a limit $u_{c,0}$.
\item[(iii)] Solve \eqref{inls} with initial data $u_{c,0}$ to obtain $u_c$, and prove the desired properties of $u_c$.
\end{itemize}

The main point is to establish (ii).  Once this is in place, step (iii) is obtained by essentially repeating the arguments of step (ii) and appealing to the small-data and stability results for \eqref{inls}; see e.g. \cite[Proposition~6.5]{FG} for the details.

The approach to establishing (ii) is to expand the sequence $u_{n,0}$ in a linear profile decomposition adapted to the $S(\dot H^{s_c})$ Strichartz estimate.  This means that the $u_{n,0}$ may be written as a linear combination of fixed profiles, translated in space-time, plus a remainder term that becomes small in the Strichartz norm.  Convergence in $H^1$ holds provided there is only one profile (with no space-time translation) and the remainder tends to zero in $H^1$-norm as well. 

The precise result we need is the following proposition.

\begin{proposition}[Linear profile decomposition]\label{P:LPD} 
Let $\{\phi_n\}$ be a bounded sequence in $H^1$.  Then for every $M$, there exist profiles $\{\psi^j\}_{j=1}^M\subset H^1$, time shifts $t_n^j$, translation parameters $x_n^j$, and remainders $W_n^M$ so that (passing to a subsequence in $n$): 
\[
\phi_n = \sum_{j=1}^M e^{-it_n^j\Delta} \psi^j(x-x_n^j) + W_n^M
\]
with the following properties:
\begin{itemize}
\item Orthogonality of parameters: for $j\neq k$, 
\[
|t_n^j-t_n^k|+|x_n^j-x_n^k| \to \infty \qtq{as}n\to\infty.
\]
\item Vanishing of the remainder:
\[
\limsup_{M\to\infty}\limsup_{n\to\infty} \|e^{it\Delta} W_n^M \|_{S(\dot H^{s_c})} = 0.
\]
\item Energy decoupling: for any $M$ and any $s\in[0,1]$,
\[
\|\phi_n\|_{\dot H^{s}}^2 = \sum_{j=1}^M\|\psi^j\|_{\dot H^s}^2 + \|W_n^M\|_{\dot H^s}^2 + o_n(1)\qtq{as}n\to\infty.
\]
\end{itemize}
Finally, we may assume either $t_n^j\equiv 0$ or $t_n^j\to\pm\infty$, and either $x_n^j\equiv 0$ or $|x_n^j|\to\infty$. 
\end{proposition}

Similar decompositions now appear in many works, beginning with some fundamental results in \cite{BG, BV, Keraani, CK, MV}. The analogue of Proposition~\ref{P:LPD} for radial sequences appears as Proposition~6.1 in \cite{FG}; a non-radial version can be found in \cite{Shao}, for example.  In the setting of \cite{FG}, the radial assumption implies that the translation parameters $x_n^j$ may be taken to be identically zero.  In fact, this is the only place in their entire paper that they rely directly on the radial assumption!  We return to this point below. 

Applying the linear profile decomposition to the sequence $u_{n,0}$, we are now tasked with proving the following: 
\begin{itemize}
\item[(a)] there is a single profile $\psi$ present, 
\item[(b)] the time shifts $t_n$ obey $t_n\equiv 0$, 
\item[(c)] the translation parameters $x_n$ obey $|x_n|\equiv 0$, and 
\item[(d)] the error $W_n$ coverges to zero strongly in $H^1$. 
\end{itemize}
(Again, we remark that (c) is automatic in \cite{FG} due to the radial assumption.)

Item (a) is proven by contradiction, with the general approach as follows.  Suppose there are multiple profiles $\psi^j$.  Recalling item (i) above and using energy decoupling, we can show that each profile lives below the critical threshold (i.e. $M[\psi^j]^{1-s_c}E[\psi^j]^{s_c}<\delta_c$ and $\psi^j$ obeys \eqref{below2}).  We would then like to associate scattering solutions to \eqref{inls} to each $\psi^j$.

First, if $x_n^j\equiv 0$ and $t_n^j\equiv 0$, we take $v^j$ to be the scattering solution to \eqref{inls} with data $\psi^j$. If instead $x_n^j\equiv 0$ and $t_n^j\to\pm\infty$, we take $v^j$ to be the solution that scatters to $e^{it\Delta}\psi^j$ (cf. \cite[Proposition~5.3]{FG}). In both cases we set $v_n^j(t,x)=v^j(t+t_n^j,x)$. 

In the case of \cite{FG}, this covers all possibilities, as $|x_n^j|\equiv 0$ always holds.  Then one can define the sequence
\begin{equation}\label{NPD}
u_n^M(t)  = \sum_{j=1}^M v_n^j(t) + e^{it\Delta}W_n^M, 
\end{equation}
and immediately observe that $u_n^M$ match $u_{0,n}$ closely in $H^1$ at $t=0$ by construction.  To complete the argument, one shows that due to the orthogonality of the parameters, the functions $u_n^M$ are approximate solutions to \eqref{inls} that obey global space-time bounds.  Using the stability lemma (Lemma~\ref{L:Stab}), this implies that $u_n^M$ and the true solutions $u_n$ are close for all times, and in particular the solutions $u_n$ inherit the good bounds from the $u_n^M$.  As the $u_n$ were constructed to have diverging space-time norms, this yields the desired contradiction and completes the proof of (a).

We can now see precisely what is needed to extend the result of \cite{FG} to the non-radial setting: we need a method to construct scattering solutions to \eqref{inls} corresponding to profiles $\psi^j$ with $|x_n^j|\to\infty$. We cannot simply solve \eqref{inls} with initial data $\psi^j$ and then translate the solution by $x_n^j$, as the inhomogeneity in the nonlinearity breaks the translation invariance of the equation.  It is here that we introduce our new ingredient:

\begin{proposition}[Scattering for data living far from the origin]\label{P} Fix $\phi \in H^1$. Let $\{t_n\}$ be a sequence of times obeying $t_n\equiv 0$ or $t_n\to\pm\infty$ and $\{x_n\}$ a sequence in $\R^3$ satisfying $|x_n|\to\infty$.  Then for all $n$ sufficiently large, there exists a global solution $v_n$ to \eqref{inls} with
\[
v_n(0) = \phi_n := e^{it_n\Delta}\phi (x-x_n)
\] 
that scatters in both time directions and obeys
\[
\| v_n\|_{S(\dot H^{s_c})}+\|v_n\|_{S(L^2)}+\|\nabla v_n\|_{S(L^2)}\lesssim 1,
\]
with implicit constant depending on $\|\phi\|_{H^1}$. 

Furthermore, for $\eps>0$, there exists $N$ and $\psi\in C_c^\infty(\R\times\R^3)$ such that
\[
\| v_n - \psi(\cdot + t_n,\cdot - x_n) \|_{S(\dot H^{s_c})} < \eps\qtq{for}n\geq N. 
\]

\end{proposition}

With Proposition~\ref{P} in place, we can construct scattering solutions corresponding to profiles with $|x_n^j|\to\infty$, and we can once again construct the `nonlinear profile decomposition' \eqref{NPD}.  The rest of the argument then goes through as described above.  Note that one may need to exploit orthogonality of the $x_n^j$ rather than that of the $t_n^j$ in order to show that the $u_n^M$ are approximate solutions.  In fact, the argument is the same as the one appearing in \cite[Proof of Claim~1, p. 4218]{FG}.  As approximation by functions in $C_c^\infty(\R^{1+3})$ is needed for this step, we have included such a statement in Proposition~\ref{P}.  

Having established item (a) above (i.e. the presence of a single profile), items (b)--(d) follow quickly using either stability theory or Proposition~\ref{P}.  As complete details are provided in \cite{FG}, let us only briefly give the ideas here: (b) If the time shifts diverge, one can use stability theory (comparing $u_n$ to linear solutions) to prove that the solutions $u_n$ would obey uniform space-time bounds.  (c) Similarly, if $|x_n|\to\infty$ then Proposition~\ref{P} and the stability result would imply the same.  (d) Finally, the strong convergence of the remainder to zero follows from the fact that if the remainder captured a nontrivial amount of $H^1$-norm, then $\phi$ would be below the critical threshold and hence the solutions $u_n$ would scatter.

This completes our discussion of the proof of Proposition~\ref{P:uc} in the non-radial setting, and hence concludes the proof of Theorem~\ref{T}.  It only remains to prove Proposition~\ref{P}, which we do in the following section.

We conclude this section with a few general remarks about some related problems.  As we have discussed, in the setting of \eqref{inls}, the presence of the decaying factor $|x|^{-b}$ ultimately precludes the possibility of diverging translation parameters.  For the standard NLS, one really must contend with the possibility that such parameters are present.  In particular, in constructing the minimal blowup solution one finds that the sequence of initial data $u_{0,n}$ only converge in $H^1$ modulo translation.  When constructing the corresponding compact solution $u_c$, one then obtains that the orbit of $u_c$ is pre-compact in $H^1$ modulo translation by some time-dependent spatial center $x(t)$.  The job of the radial assumption is basically to impose $x(t)\equiv 0$, so that the localized virial argument may be applied.  To treat the non-radial NLS, the authors of \cite{DHR} made a further argument utilizing the conservation of momentum to prove that $|x(t)|=o(t)$, which provides enough control over $x(t)$ to close the localized virial argument sketched above.  For models with broken translation symmetry (and so no conserved momentum), results analogous to Proposition~\ref{P} can provide an alternate route to establishing $x(t)\equiv 0$, even in the non-radial setting (see e.g. \cite{KMVZ, LMM, KMVZZ2}). 


\section{Proof of Proposition~\ref{P}}\label{S:P}

We turn to the proof of Proposition~\ref{P}, which we reproduce here:

\begin{prop}  Fix $\phi \in H^1$. Let $\{t_n\}$ be a sequence of times obeying $t_n\equiv 0$ or $t_n\to\pm\infty$ and $\{x_n\}$ a sequence in $\R^3$ satisfying $|x_n|\to\infty$.  Then for all $n$ sufficiently large, there exists a global solution $v_n$ to \eqref{inls} with
\[
v_n(0) = \phi_n := e^{it_n\Delta}\phi (x-x_n)
\] 
that scatters in both time directions and obeys
\[
\| v_n\|_{S(\dot H^{s_c})}+\|v_n\|_{S(L^2)}+\|\nabla v_n\|_{S(L^2)}\lesssim 1,
\]
with implicit constant depending on $\|\phi\|_{H^1}$. 

Furthermore, for $\eps>0$, there exists $N$ and $\psi\in C_c^\infty(\R\times\R^3)$ such that
\[
\| v_n - \psi(\cdot + t_n,\cdot - x_n) \|_{S(\dot H^{s_c})} < \eps\qtq{for}n\geq N. 
\]

\end{prop}

\begin{proof} We introduce a sequence of smooth cutoffs $\chi_n$ obeying
\[
\chi_n(x) = \begin{cases} 1 & |x+x_n| >\tfrac12 |x_n|, \\ 0 & |x+x_n|<\tfrac14 |x_n|,\end{cases}
\]
with $\chi_n$ obeying the symbol bounds $|\partial^\alpha \chi_n| \lesssim |x_n|^{-|\alpha|}$ for all multiindices $\alpha$.  In particular, we have $\chi_n\to 1$ pointwise as $n\to\infty$. 

We next define a family of approximations $\tilde v_{n,T}$ parametrized both by $n$ and by times $T>0$.  First, we let
\[
\tilde v_{n,T}(t,x) = \chi_n(x-x_n) e^{it\Delta} P_n \phi(x-x_n)\qtq{for} |t|\leq T,
\]
where we have set
\[
P_n = P_{\leq |x_n|^\theta}\qtq{for some small}0<\theta\ll 1.  
\]
Next, for $|t|>T$ we take the free evolution:
\[
\tilde v_{n,T}(t) = \begin{cases} e^{i(t-T)\Delta}[\tilde v_{n,T}(T)] & t> T, \\
e^{i(t+T)\Delta}[\tilde v_{n,T}(-T)], & t < -T. \end{cases}
\]

Our goal is to prove that (for sufficiently large $n$ and $T$)  the $\tilde v_{n,T}$ are approximate solutions to \eqref{inls} obeying global space-time bounds, with initial data close to $\phi_n$.  Once we have shown this, we can apply the stability result (Lemma~\ref{L:Stab}) to deduce the existence of scattering solutions to \eqref{inls} with initial data $\phi_n$, as desired. 

\begin{remark}\label{Explain} We would like to pause and explain the logic of designing the approximate solutions in this way.  The basic idea is that since the profiles $\phi_n$ are being translated far away from zero, the nonlinear term (containing $|x|^{-b}$) should essentially become negligible, and so we expect that we can approximate a solution to \eqref{inls} by a solution to the linear Schr\"odinger equation.  The role of the cutoff $\chi_n$ is to make this assertion precise (cf. the estimate of \eqref{E1} below).  However, the insertion of a spatial cutoff means that the $\tilde v_{n,T}$ are no longer true solutions to the Schr\"odinger equation.  In particular, when computing the errors (i.e. $(i\partial_t + \Delta) v + |x|^{-b} |v|^2 v$), we will have to contend with error terms that are linear in $v$, which arise when derivatives land on the cutoff function. Because these error terms must be integrated in time, we are ultimately led to bounds that grow with the length of the time interval.  This means that we should only include the cutoff on a finite time interval $[-T,T]$ and look for smallness in the regime $|t|>T$ by other means.  In particular, in the long-time regime we take $\tilde v_{n,T}$ to be a true solution to the linear Schr\"odinger equation, and the smallness as $T\to\infty$ is obtained by Strichartz estimates combined with the monotone convergence theorem.  Finally, the role of the frequency projection arises from the fact that our stability lemma (Lemma~\ref{L:Stab}) demands control over one derivative of the error in space-time norms, leading to error terms of the form $\nabla(\nabla \chi \cdot \nabla \phi)$.  As we only know $\phi \in H^1$, we are therefore forced to truncate $\phi$ in frequency.  As we still need to obtain $\phi$ in the $n\to\infty$ limit, we use a slowly growing frequency cutoff (specifically, to frequencies below $|x_n|^{\theta}$). With this choice, the losses that come from estimating this term via Bernstein's inequality can be overcome using other terms that come with negative powers of $|x_n|$.  
\end{remark}

Let us first establish closeness of the initial data:
\begin{equation}\label{ONE}
\lim_{T\to\infty}\limsup_{n\to\infty} \| \tilde v_{n,T}(t_n)-\phi_n\|_{H^1} = 0. 
\end{equation}

\begin{proof}[Proof of \eqref{ONE}] First suppose $t_n\equiv 0$.  Then
\[
\| \tilde v_{n,T}(t_n) - \phi_n \|_{H^1}  = \|\chi_n P_n\phi-\phi\|_{H^1}\to 0 \qtq{as}n\to\infty
\]
by the dominated convergence theorem. Suppose instead $t_n\to\infty$ and fix $T>0$.  Then for $n$ sufficiently large,
\[
\tilde v_{n,T}(t_n) = e^{i(t_n-T)\Delta}\chi_n(x-x_n)e^{iT\Delta}P_n\phi(x-x_n),
\]
and hence
\[
\| \tilde v_{n,T}(t_n) - \phi_n \|_{H^1}  \leq \| P_n\phi - \phi\|_{H^1}  + \| [\chi_n-1] e^{iT\Delta} P_n\phi\|_{H^1}, 
\]
which again tend to zero by dominated convergence.  The case $t_n\to-\infty$ is similar, and hence we complete the proof of \eqref{ONE}.
\end{proof}

We next prove global space-time bounds for the functions $\tilde v_{n,T}$.

\begin{equation}\label{TWO}
\limsup_{T\to\infty}\limsup_{n\to\infty}\bigl\{ \|\tilde v_{n,T}\|_{L_t^\infty H_x^1} + \|\tilde v_{n,T}\|_{S(\dot H^{s_c})} \bigr\} \lesssim 1,
\end{equation}
where all space-time norms are over $\R\times\R^3$.

\begin{proof}[Proof of \eqref{TWO}] Once we have uniform $H^1$ bounds on $[-T,T]$, all of the desired bounds on $\{|t|>T\}$ follow from Strichartz.  Thus, we may restrict our attention to $\{|t|\leq T\}$.  In this range, the desired $L^2$ bounds are immediate, while
\begin{align*}
\| \nabla \tilde v_{n,T}\|_{L_t^\infty L_x^2} & \lesssim \| \nabla (\chi_n)\|_{L_x^3} \| \phi\|_{L_x^6} + \|\chi_n\|_{L_x^\infty} \|\nabla \phi\|_{L_x^2} \lesssim \|\phi\|_{H^1}
\end{align*}
by Sobolev embedding and the properties of $\chi_n$.  Similarly, we can establish $L_t^q L_x^{r}$ bounds for any $(q,r)\in \mathcal{A}_{s_c}$ immediately from Sobolev embedding and Strichartz estimates. This completes the proof of \eqref{TWO}. 
\end{proof} 

Finally, we need to prove that the $\tilde v_{n,T}$ define good approximate solutions to \eqref{inls}.  We define the errors
\[
e_{n,T} = (i\partial_t + \Delta)\tilde v_{n,T} + |x|^{-b} |\tilde v_{n,T}|^2 \tilde v_{n,T},
\]
and we will show:
\begin{equation}\label{THREE}
\lim_{T\to\infty}\limsup_{n\to\infty} \bigl\{\|  e_{n,T} \|_{S'(L^2)} + \|\nabla  e_{n,T} \|_{S'(L^2)} + \|  e_{n,T} \|_{S'(\dot H^{-s_c})} \bigr\} = 0.
\end{equation}

\begin{proof}[Proof of \eqref{THREE}]  We first consider the region $t>T$, with the region $t<-T$ being treated by similar arguments. In this region
\[
e_{n,T} = |x|^{-b} |\tilde v_{n,T}|^2 \tilde v_{n,T}. 
\]
We will show that
\begin{align*}
\lim_{T\to\infty}\limsup_{n\to\infty}\bigl\{ & \|e_{n,T}\|_{L_t^{1} L_x^{2}(\{t>T\})} + \|\nabla e_{n,T}\|_{L_t^{1} L_x^{2}(\{t>T\})} \\
& \quad+ \|e_{n,T}\|_{L_t^{\frac{4}{3+b}-} L_x^{\frac{6}{5}+}(\{t>T\})}\bigr\} = 0.
\end{align*}
The last norm corresponds essentially to the endpoint of the admissible region for the dual Strichartz norm (which appears in the stability result, Lemma~\ref{L:Stab}).  In most instances below, we will simply estimate the endpoint, since the arguments we give always allow the spaces to be perturbed slightly.  It is only in the estimation of \eqref{E3} below that we need to avoid the exact endpoint. 

We begin by using H\"older's inequality, Sobolev embedding, and Strichartz to estimate
\begin{align*}
\| |x|^{-b} |\tilde v_{n,T}|^2 \tilde v_{n,T} \|_{L_t^1 L_x^2(\{t>T\})} & \lesssim \| |x|^{-b}\|_{L_x^{\frac{3}{b},\infty}} \| \tilde v_{n,T}\|_{L_t^3 L_x^{\frac{18}{3-2b},6}(\{t>T\})}^3 \\
& \lesssim \| |\nabla|^{\frac{1+b}{3}} \tilde v_{n,T} \|_{L_t^3 L_x^{\frac{18}{5},6}(\{t>T\})}^3 \\
&\lesssim \| |\nabla|^{\frac{1+b}{3}} e^{it\Delta}[\tilde v_{n,T}(T)]\|_{L_t^3 L_x^{\frac{18}{5}}(\{t>0\})}^3.
\end{align*}
Now we recall the definition of $\tilde v_{n,T}(T)$ and estimate the final norm as follows:
\begin{align*}
\| |\nabla|^{\frac{1+b}{3}}e^{it\Delta} \chi_n P_n e^{iT\Delta}\phi \|_{L_t^3 L_x^{\frac{18}{5}}(\{t>0\})} & \lesssim \| (\chi_n-1) P_n \phi \|_{\dot H_x^{\frac{1+b}{3}}} \\
& \quad + \| |\nabla|^{\frac{1+b}{3}}e^{it\Delta} \phi \|_{L_t^3 L_x^{\frac{18}{5}}(\{t>T\})}.
\end{align*}
The first term above tends to zero as $n\to\infty$ by dominated convergence.  The second term is bounded by $\phi$ in $\dot H^{\frac{1+b}{3}}$ and hence the norm tends to zero as $T\to\infty$ by monotone convergence.

Next, we consider the term in \eqref{THREE} with the derivative.  This leads to two terms, one of the form $|x|^{-b}\mathcal{O}( v^2 \nabla v)$ and one of the form $|x|^{-b} \mathcal{O}( v^2 |x|^{-1} v)$.  By Hardy's inequality, we can treat these terms identically, provided we work in a space below $L_x^3$ for $|x|^{-1} v$ and $\nabla v$.  In particular, choosing $4<q<\tfrac{2}{b}$, we use Sobolev embedding and Strichartz to estimate (on the region $\{t>T\}$):
\begin{align*}
\| \nabla [|x|^{-b} |\tilde v_{n,T}|^2 \tilde v_{n,T}] \|_{L_t^1 L_x^2} & \lesssim \| |x|^{-b} \|_{L_x^{\frac{3}{b},\infty}} \|\tilde v_{n,T}\|_{L_t^{\frac{2q}{q-1}} L_x^{\frac{6q}{2-bq},3q}}^2 \| \nabla \tilde v_{n,T} \|_{L_t^{q} L_x^{\frac{6q}{3q-4}}} \\
& \lesssim \| |\nabla|^{\frac{1+b}{2}} \tilde v_{n,T} \|_{L_t^{\frac{2q}{q-1}} L_x^{\frac{6q}{q+2},3q}}^2 \| \tilde v_{n,T}(T)\|_{H_x^1} \\
& \lesssim \| |\nabla|^{\frac{1+b}{2}} e^{it\Delta}[\tilde v_{n,T}(T)] \|_{L_t^{\frac{2q}{q-1}} L_x^{\frac{6q}{q+2}}(\{t>0\})}^2
\end{align*}
uniformly in $n,T$.  Noting that $\tfrac{1+b}{2}<1$, we find that we are in the same position as above, and so we may estimate as before to conclude that this term tends to zero as $n,T\to\infty$.  

We now consider the final norm over the region $\{t>T\}$. Using Sobolev embedding, we get
\begin{align*}
\| |x|^{-b} |\tilde v_{n,T}|^2 \tilde v_{n,T}\|_{L_t^{\frac{4}{3+b}} L_x^{\frac65}} & \lesssim \| |x|^{-b} \|_{L_x^{\frac{3}{b},\infty}} \| \tilde v_{n,T}\|_{L_t^{\frac{12}{3+b}} L_x^{\frac{18}{5-2b},\frac{18}{5}}}^3 \\
& \lesssim \| |\nabla|^{\frac{1+b}{6}}\tilde v_{n,T}\|_{L_t^{\frac{12}{3+b}} L_x^{\frac{18}{6-b},\frac{18}{5}}}^3 \\
& \lesssim  \| |\nabla|^{\frac{1+b}{6}} e^{it\Delta}[\tilde v_{n,T}(T)]\|_{L_t^{\frac{12}{3+b}} L_x^{\frac{18}{6-b}}(\{t>0\})}^3.
\end{align*}
Once again, we are in a similar situation to the ones encountered above, and so the same analysis suffices to show that this term tends to zero as $n,T\to\infty$. 

It remains to consider the region $|t|\leq T$ in \eqref{THREE}.   In this region we can compute
\begin{align}
e_{n,T}(t,x) & = |x|^{-b}\chi_n^3(x-x_n) |e^{it\Delta}P_n\phi(x-x_n)|^2 e^{it\Delta}P_n\phi(x-x_n) \label{E1} \\
& + \Delta[\chi_n(x-x_n)] e^{it\Delta} P_n\phi(x-x_n) \label{E2} \\
& + 2\nabla[\chi_n(x-x_n)]\cdot \nabla e^{it\Delta} P_n\phi(x-x_n). \label{E3}
\end{align}
We will show that for fixed $T$, each of these terms tends to zero as $n\to\infty$. 

Let us first consider the contribution of \eqref{E1}.  On the support of this term, we have the pointwise estimate $|x|^{-b} \lesssim |x_n|^{-b}$.  Thus we may bound
\begin{align*}
\| \eqref{E1}\|_{L_t^1 L_x^2} & \lesssim |x_n|^{-b} \| e^{it\Delta} P_n\phi \|_{L_t^3 L_x^6}^3 \\
& \lesssim |x_n|^{-b} \| |\nabla|^{\frac13} e^{it\Delta} \phi\|_{L_t^3 L_x^{\frac{18}{5}}}^3 \\
& \lesssim |x_n|^{-b} \|\phi\|_{H^1}^3 \to 0 \qtq{as}n\to\infty. 
\end{align*}
We next consider the derivative of this quantity.  If the derivative lands on $|x|^{-b}$ or on the cutoff, we can estimate exactly as above, attaining the bound $|x_n|^{-b-1}$ instead of $|x_n|^{-b}$.  If instead the derivative lands on a copy of the free solution, we could either rearrange the spaces slightly, or we can recall that $P_n \phi$ is frequency localized to frequencies $\leq |x_n|^\theta$.  Thus the extra derivative would ultimately contribute $|x_n|^\theta$ to the estimate above (via Bernstein's inequality), which is acceptable provided we choose $\theta<b.$  Finally, we can estimate the remaining space-time norm via
\begin{align*}
\|\eqref{E1}\|_{L_t^{\frac{4}{3+b}}L_x^{\frac65}} & \lesssim |x_n|^{-b} \| e^{it\Delta}P_n \phi \|_{L_t^{\frac{12}{3+b}}L_x^{\frac{18}{5}}}^3 \\
& \lesssim |x_n|^{-b} \| |\nabla|^{\frac{1-b}{6}} e^{it\Delta} P_n \phi \|_{L_t^{\frac{12}{3+b}} L_x^{\frac{18}{6-b}}}^3 \\
& \lesssim |x_n|^{-b}\|\phi\|_{H^1}^3 \to 0 \qtq{as}n\to\infty.  
\end{align*}

We turn to \eqref{E2} and \eqref{E3}.  The $L_t^1 L_x^2$-norm is estimated by 
\begin{align*}
T\{ |x_n|^{-2} + |x_n|^{-1}\}\|\phi\|_{H^1},
\end{align*}
which is acceptable.  For the $L_t^1 L_x^2$-norm of the derivative, we are led instead to
\begin{align*}
T\{ |x_n|^{-3} & + |x_n|^{-2} + |x_n|^{-1} \}\bigl\{ \| \phi_n\|_{L^2}+\|\nabla\phi_n\|_{L^2}+\|\Delta \phi_n\|_{L^2}\bigr\} \\
& \lesssim  T|x_n|^{-1+\theta}\|\phi\|_{H^1}^3 \to 0 \qtq{as}n\to\infty. 
\end{align*}
Finally, we consider the $L_t^{\frac{4}{3+b}}L_x^{\frac65}$-norm.  For \eqref{E2}, we use H\"older's inequality to estimate
\[
T^{\frac{3+b}{4}}\| \Delta \chi_n\|_{L^3}\|\phi\|_{L^2} \lesssim T^{\frac{3+b}{4}}|x_n|^{-1}\|\phi\|_{L^2}\to 0 \qtq{as}n\to\infty. 
\]
For \eqref{E3}, we instead have
\begin{align*}
\|\nabla\chi_n\cdot \nabla e^{it\Delta}P_n\phi\|_{L_t^{\frac{4}{3+b}-}L_x^{\frac65+}}&  \lesssim T^{\frac{3+b}{4}+} \| \nabla \chi_n\|_{L_x^{3+}} \|\nabla \phi\|_{L^2}\\
&  \lesssim T^{\frac{3+b}{4}+}|x_n|^{0-}\|\phi\|_{H^1} \to 0 \qtq{as}n\to\infty.
\end{align*}
This completes the proof of \eqref{THREE} in the regime $|t|\leq T$.  \end{proof}

Having established \eqref{ONE}, \eqref{TWO}, and \eqref{THREE}, we can now appeal to the stability result, Lemma~\ref{L:Stab} to deduce the existence of a global solution $v_n$ to \eqref{inls} satisfying $v_n(0)=\phi_n$ and obeying 
\[
\|v_n\|_{S(\dot H^{s_c})} + \|v_n\|_{S(L^2)} + \|\nabla v_n\|_{S(L^2)}\lesssim 1
\]
for all $n$ sufficiently large.

It remains to establish the approximation by $C_c^\infty$ functions. We first observe that the construction above yields
\[
\lim_{T\to\infty} \limsup_{n\to\infty} \| v_n(\cdot-t_n)-\tilde v_{n,T}(\cdot)\|_{S(\dot H^{s_c})} = 0. 
\]
Given $\eps>0$, we choose $\psi\in C_c^\infty(\R^{1+3})$ so that
\[
\|e^{it\Delta} \phi - \psi\|_{S(\dot H^{s_c})} <\eps,
\]
which then reduces the problem to proving
\[
\| \tilde v_{n,T}(t,x) - e^{it\Delta}\phi(x-x_n)\|_{S(\dot H^{s_c})}<\eps
\]
for $n,T$ large. We consider the region $\{|t|\leq T\}$ and $\{|t|>T\}$ separately.  First, on $\{|t|\leq T\}$ we estimate
\begin{align*}
\|\tilde v_{n,T}(t,x)- e^{it\Delta}\phi(x-x_n)\|_{S(\dot H^{s_c})} & \lesssim \|[\chi_n-1]e^{it\Delta}\phi\|_{S(\dot H^{s_c})} + \|P_n\phi-\phi\|_{H^1} \\
& =o(1) \qtq{as}n\to\infty
\end{align*}
by dominated convergence.  For $t>T$, say, we need to estimate
\begin{align*}
\|e^{it\Delta}e^{-iT\Delta}\chi_n e^{iT\Delta}P_n\phi - e^{it\Delta}\phi\|_{S(\dot H^{s_c};(T,\infty))}.
\end{align*}
In fact, applying the triangle inequality and using monotone convergence, the problem is reduced to proving
\[
\lim_{T\to\infty}\limsup_{n\to\infty}\| e^{it\Delta}\{\chi_n e^{iT\Delta} P_n\phi\}\|_{S(\dot H^{s_c};(0,\infty))}=0. 
\]
To this end, we note that the norm above can be bounded by
\begin{align*}
\|[\chi_n-1] e^{iT\Delta} P_n \phi\|_{\dot H^{s_c}} + \|e^{it\Delta}  \phi \|_{S(\dot H^{s_c};(T,\infty))} + \|P_n\phi - \phi\|_{\dot H^{s_c}}.  
\end{align*}
Then we can see that the first and third terms tend to zero as $n\to\infty$ (by dominated convergence), while the second term can be shown to tend to zero by Strichartz and monotone convergence.  This completes the proof. \end{proof}

\end{document}